\documentclass[12pt,reqno]{amsart}
\usepackage{amsmath,amssymb,cite,geometry,bm,easybmat}
\usepackage[colorlinks,citecolor=blue,linkcolor=blue,urlcolor=blue]{hyperref}
\newtheorem{theorem}{Theorem}[section]
\newtheorem{lemma}[theorem]{Lemma}
\newtheorem{proposition}[theorem]{Proposition}

\newtheorem{remark}[theorem]{Remark}
\numberwithin{equation}{section}
\allowdisplaybreaks[4]

\begin{document}

\title[Determining coefficients of thermoelastic system]{Determining coefficients of thermoelastic system from boundary information}

\author{Xiaoming Tan}
\address{School of Mathematics and Statistics, Beijing Institute of Technology, Beijing 100081, China}
\email{xtan@bit.edu.cn}

\subjclass[2020]{35R30, 74F05, 74E05, 58J32, 58J40}

\keywords{Thermoelastic system; Thermoelastic Calder\'{o}n's problem; Thermoelastic Dirichlet-to-Neumann map; Inverse problems; Pseudodifferential operators.\\
{\bf ---------------}\\
Xiaoming Tan\\
{\it School of Mathematics and Statistics, Beijing Institute of Technology, Beijing 100081, China.}\\
{\it Email address}: xtan@bit.edu.cn}

\begin{abstract}
    Given a compact Riemannian manifold $(M,g)$ with smooth boundary $\partial M$, we give an explicit expression for full symbol of the thermoelastic Dirichlet-to-Neumann map $\Lambda_g$ with variable coefficients $\lambda,\mu,\alpha,\beta \in C^{\infty}(\bar{M})$. We prove that $\Lambda_g$ uniquely determines partial derivatives of all orders of the coefficients on the boundary. Moreover, for a nonempty open subset $\Gamma\subset\partial M$, suppose that the manifold and the coefficients are real analytic up to $\Gamma$, we show that $\Lambda_g$ uniquely determines the coefficients on the whole manifold $\bar{M}$.
\end{abstract}

\maketitle 

\section{Introduction}

\addvspace{5mm}

In this paper, we will study the  thermoelastic Calder\'{o}n problem, that is, whether one can uniquely determine the Lam\'{e} coefficients $\lambda,\mu$ and the other two physical coefficients $\alpha,\beta$ of a thermoelastic body by boundary information? Let $(M,g)$ be a compact Riemannian manifold of dimension $n$ with smooth boundary $\partial M$. We consider the manifold $M$ as an inhomogeneous, isotropic, thermoelastic body. Assume that the coefficient $\beta \in C^{\infty}(\bar{M})$, the Lam\'{e} coefficients $\lambda,\mu \in C^{\infty}(\bar{M})$ and the heat conduction coefficient $\alpha\in C^{\infty}(\bar{M})$ of the thermoelastic body satisfy $\mu > 0, \lambda + \mu \geqslant 0$ and $\alpha>0$.

For the displacement vector field $\bm{u} \in (C^{\infty}( M))^{n}$ and the temperature variation $\theta \in C^{\infty}( M)$, we define the thermoelastic operator $T_g$ with variable coefficients as (cf. \cite{Liu19,LiuTan22.2,TanLiu22,Kupr80})
\begin{align}\label{1.2}
    T_g  
    \begin{bmatrix}
        \bm{u}\\
        \theta
    \end{bmatrix}:=
    \begin{bmatrix}
        L_g + \rho\omega^2 & -\beta\operatorname{grad} \\[2mm]
        i\omega\theta_{0}\beta\operatorname{div} & \alpha\Delta_{g} + i\omega \gamma
    \end{bmatrix}
    \begin{bmatrix}
        \bm{u}\\
        \theta
    \end{bmatrix},
\end{align}
where the Lam\'{e} operator $L_g$ with variable coefficients is defined by (see \cite{TanLiu22})
\begin{align}\label{1.3}
    L_g \bm{u} 
    &:= \mu \Delta^{}_{B} \bm{u} + (\lambda + \mu)\operatorname{grad}\operatorname{div} \bm{u} + \mu \operatorname{Ric}(\bm{u}) \notag\\
    &\qquad + (\operatorname{grad} \lambda) \operatorname{div} \bm{u} + (S\bm{u})(\operatorname{grad} \mu).
\end{align}
Here we denote by $\operatorname{grad},\operatorname{div},\Delta_{g},\Delta_{B}$ and $\operatorname{Ric}$, respectively, the gradient operator, the divergence operator, the Laplace--Beltrami operator, the Bochner Laplacian and the Ricci tensor with respect to the metric $g$. The stress tensor $S$ (also called deformation tensor) of type $(1,1)$ is defined by (see \cite[p.\,562]{Taylor11.3})
\begin{align*}
    S\bm{u}:=\nabla \bm{u} + \nabla \bm{u}^t,
\end{align*}
the coefficient $\beta\in C^{\infty}(\bar{M})$ depends on Lam\'{e} coefficients and linear expansion coefficient of the thermoelastic body, $\gamma$ is the specific heat per unit volume, $\theta_{0}$ is the reference temperature, $\rho$ is the density of the thermoelastic body, $\omega$ is the angular frequency and $i = \sqrt{-1}$. In particular, the Lam\'{e} operator with constant coefficients has the form $L\textbf{\textit{u}}=\mu\Delta\textbf{\textit{u}}+(\lambda+\mu)\nabla(\nabla\cdot\textbf{\textit{u}})$ in Euclidean domains (see \cite{Kupr80,LandLif86}).

\addvspace{3mm}

We consider the following Dirichlet boundary value problem for the thermoelastic system
\begin{equation}\label{1.1}
    \begin{cases}
        T_g  \bm{U} = 0 \quad & \text{in}\  M, \\
        \bm{U} = \bm{V} & \text{on}\ \partial  M,
    \end{cases}
\end{equation}
where $\textbf{\textit{U}} = (\textbf{\textit{u}},\theta)^t$ and the superscript $t$ denotes the transpose. Problem \eqref{1.1} is an extension of the boundary value problem for classical elastic system. In particular, when $ M$ is a bounded Euclidean domain and the temperature is not taken into consideration, problem \eqref{1.1} reduces to the corresponding problem for classical elastic system.

For any boundary value $\bm{V}\in (H^{1/2}(\partial M))^{n+1}$, there is a unique solution $\bm{U}$ solves the above problem \eqref{1.1} by the theory of elliptic operators. Therefore, we define the thermoelastic Dirichlet-to-Neumann map $\Lambda_g :(H^{1/2}(\partial M))^{n+1} \to (H^{-1/2}(\partial M))^{n+1}$ associated with the thermoelastic operator $T_g$ as (see \cite{LiuTan22.2})
\begin{align}\label{1.4}
    \Lambda_g  (\bm{U}|_{\partial  M}) :=
    \begin{bmatrix}
        \lambda\nu \operatorname{div} + \mu\nu S & -\beta \nu \\[2mm]
        0 & \alpha\partial_\nu
    \end{bmatrix}
    \bm{U} \quad \text{on}\ \partial  M,
\end{align}
where $\nu$ is the outward unit normal vector to the boundary $\partial M$. The thermoelastic Dirichlet-to-Neumann map $\Lambda_g$ is an elliptic, self-adjoint pseudodifferential operator of order one defined on the boundary. In this paper, we will study the thermoelastic Calder\'{o}n problem on a Riemannian manifold, which is determining the coefficients $\lambda,\mu,\alpha,\beta\in C^{\infty}(\bar{M})$ by the thermoelastic Dirichlet-to-Neumann map $\Lambda_g$. By giving explicit expressions for $\Lambda_g$ and its full symbol $\sigma(\Lambda_g)$, we show that $\Lambda_g$ uniquely determines the coefficients $\lambda,\mu,\alpha,\beta$.

\addvspace{3mm}

We briefly recall some uniqueness results for the classical Calder\'{o}n problem and the elastic Calder\'{o}n problem. The classical Calder\'{o}n problem \cite{Cald80}: whether one can uniquely determine the electrical conductivity of a medium by making voltage and current measurements at the boundary of the medium? This problem has been studied for decades. For a bounded Euclidean domain $\Omega\subset\mathbb{R}^n$ with smooth boundary $\partial\Omega$, $n\geqslant 2$, Kohn and Vogelius \cite{KohnVoge84} proved a famous uniqueness result on the boundary for $C^{\infty}$-conductivities, that is, if $\Lambda_{\gamma_1}=\Lambda_{\gamma_2}$, then $\frac{\partial^{|J|} \gamma_1}{\partial x^J}\big|_{\partial \Omega}=\frac{\partial^{|J|} \gamma_2}{\partial x^J}\big|_{\partial \Omega}$ for all multi-indices $J\in\mathbb{N}^n$. This settled the uniqueness problem on the boundary in the real analytic category. They extended the uniqueness result to piecewise real analytic conductivities in \cite{KohnVoge85}. In dimensions $n\geqslant 3$, in the celebrated paper \cite{SylvUhlm87} Sylvester and Uhlmann proved the uniqueness of the $C^{\infty}$-conductivities by constructing the complex geometrical optics solutions. The classical Calder\'{o}n problem have attracted lots of attention for decades (see, for example, \cite{AstPai06,Nachman96,CaroRogers16,Haberata13,KohnVoge84,KohnVoge85,Nachman96,SylvUhlm87,ImaUhlmYama10,ALP05,SunUhl03,SunUhl97,LeeUhlm89} and references therein). We also refer the reader to the survey articles \cite{Uhlm09,Uhlm14} for the classical Calder\'{o}n problem and related topics.

For the elastic Calder\'{o}n problem, partial uniqueness results for determination of Lam\'{e} coefficients from boundary measurements were obtained. For a bounded Euclidean domain $\Omega\subset\mathbb{R}^n$ with smooth boundary $\partial\Omega$, Nakamura and Uhlmann \cite{NakaUhlm95} proved that one can determine the full Taylor series of Lam\'{e} coefficients on the boundary in all dimensions $n\geqslant 2$ and for a generic anisotropic elastic tensor in two dimensions. In \cite{ImaYama15} Imanuvilov and Yamamoto also proved the global uniqueness of the Lam\'{e} coefficients $\lambda,\mu \in C^{10}(\bar{\Omega})$. In three dimensional Euclidean domains, Nakamura and Uhlmann \cite{NakaUhlm94,NakaUhlm03} as well as Eskin and Ralston \cite{EskinRalston02} proved the global uniqueness of Lam\'{e} coefficients provided that $\nabla\mu$ is small in a suitable norm. However, in dimensions $n\geqslant 3$, the global uniqueness of the Lam\'{e} coefficients $\lambda,\mu\in C^{\infty}(\bar{\Omega})$ without the smallness assumption ($\|\nabla\mu\| < \varepsilon_0$ for some small positive $\varepsilon_0$) remains an open problem (see \cite[p.\,210]{Isakov17}). We also refer the reader to \cite{ANS91,NakaUhlm93,ImaYama11,ImaUhlmYama12} for the elastic Calder\'{o}n problem.

Recently, Tan and Liu \cite{TanLiu22} gave an explicit expression for full symbol of the elastic Dirichlet-to-Neumann map on a Riemannian manifold $M$ and showed that the elastic Dirichlet-to-Neumann map uniquely determines partial derivatives of all orders of the Lam\'{e} coefficients on the boundary. Moreover, for a nonempty open subset $\Gamma\subset\partial M$, suppose that the manifold and the Lam\'{e} coefficients are real analytic up to $\Gamma$, they proved that the elastic Dirichlet-to-Neumann map uniquely determines the Lam\'{e} coefficients on the whole manifold $\bar{M}$.

In mathematics, physics and engineering, there are lots of inverse problems have been studied for decades. Here we do not list all the references about these topics. We refer the reader to \cite{Liu19.2,Pichler18,CaroZhou14,JoshiMcDowall00,McDowall97,OlaPai93} for Maxwell's equations, to \cite{Liu20,HeckWangLi07,LiWang07,DKSU09,DKSU07,KrupUhlm18,KrupUhlm14,NSU95,PSU10} for incompressible fluid and many others. For the studies about other types of Dirichlet-to-Neumann map, we also refer the reader to \cite{Liu11,Liu15,Liu14,LiuTan21,LiuTan22.2} and references therein.

\addvspace{3mm}

For the sake of simplicity, we denote by $I_n$ the $n\times n$ identity matrix,
\begin{align*}
    [a^\alpha_\beta]:=
    \begin{bmatrix}
        a^1_1 & \dots & a^1_{n-1} \\
        \vdots & \ddots & \vdots \\
        a^{n-1}_1 & \dots & a^{n-1}_{n-1} \\
    \end{bmatrix},
\end{align*}
and
\begin{align*}
    \begin{bmatrix}
        [a^j_k] & [b^j] \\[2mm]
        [c_k] & d
    \end{bmatrix}
    :=
    \begin{bmatrix}
        [a^\alpha_\beta] & [a^\alpha_n] & [b^\alpha] \\[2mm]
        [a^n_\beta] & a^n_n & b^n \\[2mm]
        [c_\beta] & c_n & d
    \end{bmatrix}
    =
    \begin{bmatrix}
        a^1_1 & \dots & a^1_n & b^1 \\
        \vdots & \ddots & \vdots & \vdots \\
        a^n_1 & \dots & a^n_n & b^n \\
        c_1 & \dots & c_n & d \\
    \end{bmatrix},
\end{align*}
where $ 1 \leqslant \alpha,\beta \leqslant n-1$ and $1 \leqslant j,k \leqslant n$.

\addvspace{5mm}

The main results of this paper are the following three theorems.
\begin{theorem}\label{thm1.3}
    Let $(M,g)$ be a compact Riemannian manifold of dimension $n$ with smooth boundary $\partial M$. Assume that the coefficient $\beta\in C^{\infty}(\bar{M})$, the Lam\'{e} coefficients $\lambda,\,\mu \in C^{\infty}(\bar{M})$ and the heat conduction coefficient $\alpha\in C^{\infty}(\bar{M})$ satisfy $\mu > 0,\lambda + \mu \geqslant 0$ and $\alpha>0$. Let $\sigma(\Lambda_g ) \sim \sum_{j\leqslant 1} p_j(x,\xi^{\prime})$ be the full symbol of the thermoelastic Dirichlet-to-Neumann map $\Lambda_g$. Then
    \begin{align}
        \label{18} p_1&=
        \begin{bmatrix}
            \mu|\xi^{\prime}|I_{n-1} + \frac{\mu(\lambda+\mu)}{(\lambda+3\mu)|\xi^{\prime}|} [\xi^\alpha\xi_\beta] & -\frac{2i\mu^2}{\lambda+3\mu} [\xi^\alpha] & 0 \\[2mm]
            \frac{2i\mu^2}{\lambda+3\mu}[\xi_\beta] & \frac{2\mu(\lambda+2\mu)}{\lambda+3\mu}|\xi^{\prime}| & 0 \\[2mm]
            0 & 0 & \alpha |\xi^{\prime}|
        \end{bmatrix},\\[2mm]
        \label{19} p_0&=
        \begin{bmatrix}
            \mu I_{n-1} &0 &0  \\
            0& \lambda+2\mu &0 \\
            0& 0& \alpha
        \end{bmatrix}
        q_0-
        \begin{bmatrix}
            0 & 0 &0 \\
            \lambda [\Gamma^\alpha_{\alpha\beta}] & \lambda \Gamma^\alpha_{\alpha n} & -\beta\\
            0 & 0 &0
        \end{bmatrix},\\[2mm]
        \label{20} p_{-m}&=
        \begin{bmatrix}
            \mu I_{n-1} &0 &0  \\
            0& \lambda+2\mu &0 \\
            0& 0& \alpha
        \end{bmatrix}
        q_{-m},\quad m\geqslant 1,
    \end{align}
    where $i=\sqrt{-1}$, $\xi^{\prime}=(\xi_1,\dots,\xi_{n-1})$, $\xi^\alpha=g^{\alpha\beta}\xi_\beta$, $|\xi^{\prime}|=\sqrt{\xi^\alpha\xi_\alpha}$, and $q_{-m}\,(m\geqslant 0)$ are given by \eqref{3.1.1} in Section $\ref{s2}$.
\end{theorem}

\addvspace{3mm}

For the case of the thermoelastic Dirichlet-to-Neumann map with constant coefficients on a Riemannian manifold, the corresponding full symbol had been obtained in \cite{LiuTan22.2}. For the case of the elastic Dirichlet-to-Neumann map constant coefficients, the corresponding full symbol had been obtained in \cite{Liu19}. The principal symbol of the thermoelastic Dirichlet-to-Neumann map had also be studied in \cite{Vodev22} and \cite{ZhangY20} in the context of the thermoelastic wave equations in Euclidean setting.

By studying the full symbol of the thermoelastic Dirichlet-to-Neumann map $\Lambda_g$, we prove the following result:

\begin{theorem}\label{thm1.1}
    Let $(M,g)$ be a compact Riemannian manifold of dimension $n$ with smooth boundary $\partial M$. Assume that the coefficient $\beta\in C^{\infty}(\bar{M})$, the Lam\'{e} coefficients $\lambda,\,\mu \in C^{\infty}(\bar{M})$ and the heat conduction coefficient $\alpha\in C^{\infty}(\bar{M})$ satisfy $\mu > 0,\lambda + \mu \geqslant 0$ and $\alpha>0$. Then the thermoelastic Dirichlet-to-Neumann map $\Lambda_g$ uniquely determines $\frac{\partial^{|J|} \lambda}{\partial x^J}$, $\frac{\partial^{|J|} \mu}{\partial x^J}$, $\frac{\partial^{|J|} \alpha}{\partial x^J}$ and $\frac{\partial^{|J|} \beta}{\partial x^J}$ on the boundary for all multi-indices $J$.
\end{theorem}

The uniqueness result in Theorem \ref{thm1.1} can be extended to the whole manifold for real analytic setting.

\begin{theorem}\label{thm1.2}
    Let $(M,g)$ be a compact Riemannian manifold of dimension $n$ with smooth boundary $\partial M$, and let $\Gamma\subset \partial M$ be a nonempty open subset. Suppose that the manifold is real analytic up to $\Gamma$ and the coefficients $\lambda,\mu,\alpha,\beta$ are also real analytic up to $\Gamma$ and satisfy $\mu>0$, $\lambda + \mu \geqslant 0$ and $\alpha>0$. Then the thermoelastic Dirichlet-to-Neumann map $\Lambda_g$ uniquely determines $\lambda$, $\mu$, $\alpha$ and $\beta$ on $\bar{M}$.
\end{theorem}

\addvspace{3mm}

Theorem \ref{thm1.2} shows that the global uniqueness of real analytic coefficients on a real analytic Riemannian manifold. To the best of our knowledge, this is the first global uniqueness result for variable coefficients in thermoelasticity on a Riemannian manifold. It is clear that Theorem {\rm\ref{thm1.2}} also holds for a real analytic bounded Euclidean domain.

\addvspace{3mm}

The main ideas of this paper are as follows. Firstly, in \cite{Liu19} Liu established a method such that one can calculate the full symbol of the elastic Dirichlet-to-Neumann map with constant coefficients. In \cite{TanLiu22}, the full symbol of the elastic Dirichlet-to-Neumann map with variable coefficients was obtained. The full symbol of the thermoelastic Dirichlet-to-Neumann map with constant coefficients was obtained in \cite{LiuTan22.2}. Combining the methods and results in \cite{Liu19,TanLiu22,LiuTan22.2} we can deal with the case for variable coefficients in thermoelasticity. Then we flatten the boundary and induce a Riemannian metric in a neighborhood of the boundary and give a local representation for the thermoelastic Dirichlet-to-Neumann map $\Lambda_g$ with variable coefficients in boundary normal coordinates, that is,
\begin{align*}
    \Lambda_g  = A\Big(-\frac{\partial }{\partial x_n}\Big)-D,
\end{align*}
where $A$ and $D$ are two matrices. We then look for the following factorization for the thermoelastic operator $T_g$, and get
\begin{align*}
    A^{-1} T_g  
    = I_{n+1}\frac{\partial^2 }{\partial x_n^2} + B \frac{\partial }{\partial x_n} + C
    = \Bigl(I_{n+1}\frac{\partial }{\partial x_n} + B - Q\Bigr)\Bigl(I_{n+1}\frac{\partial }{\partial x_n} + Q\Bigr),
\end{align*}
where $B$, $C$ are two differential operators and $Q$ is a pseudodifferential operator. As a result, we obtain the equation
\begin{align*}
    Q^2 - BQ - \Bigl[\frac{\partial }{\partial x_n},Q\Bigr] + C = 0,
\end{align*}
where $[\cdot,\cdot]$ is the commutator. Finally, we solve the full symbol equation
\begin{align*}
    \sum_{J} \frac{(-i)^{|J|}}{J !} \partial_{\xi^{\prime}}^{J}q \, \partial_{x^\prime}^{J}q - \sum_{J} \frac{(-i)^{|J|}}{J !} \partial_{\xi^{\prime}}^{J}b \, \partial_{x^\prime}^{J}q - \frac{\partial q}{\partial x_n} + c = 0,
\end{align*}
which is a matrix equation, where the sum is over all multi-indices $J$, $\xi^{\prime}=(\xi_1,\dots,\xi_{n-1})$ and $x^\prime=(x_1,\dots,x_{n-1})$. Here $b$, $c$ and $q$ are the full symbols of the operators $B$, $C$ and $Q$, respectively. Thus, we obtain the full symbol $\sigma(\Lambda_g ) \sim \sum_{j\leqslant 1} p_j(x,\xi^{\prime})$ of $\Lambda_g$ from the full symbol of $Q$. Note that computations of the full symbols of matrix-valued pseudodifferential operators are quite difficult tasks. Generally, the above full symbol equation can not be exactly solved, in other words, there is not a general formula of the solution represented by the coefficients of the matrix equation. Hence, by overcoming the difficulties of computing the symbols of pseudodifferential operators and solving the symbol equation with variable coefficients, we develop the method of the previous work \cite{Liu19,TanLiu22,LiuTan22.2} to deal with the uniqueness of variable coefficients on the Riemannian manifold in thermoelasticity. The symbols $p_j(x,\xi^{\prime})$ contain the information about the coefficients $\lambda,\mu,\alpha,\beta$ and their derivatives on the boundary, thus we can prove that they can be uniquely determined by the thermoelastic Dirichlet-to-Neumann map. Furthermore, we prove that the coefficients can be uniquely determined on the whole manifold $\bar{M}$ by the thermoelastic Dirichlet-to-Neumann map provided the manifold and coefficients are real analytic.

This paper is organized as follows. In Section \ref{s2} we give an explicit expression of the thermoelastic Dirichlet-to-Neumann map $\Lambda_g$ in boundary normal coordinates and derive a factorization of the thermoelastic operator $T_g$ with variable coefficients, then we compute the full symbols of $\Lambda_g$ and the pseudodifferential operator $Q$. In Section \ref{s3} we prove Theorem \ref{thm1.3} and Theorem \ref{thm1.1} for boundary determination. Finally, Section \ref{s4} is devoted to proving Theorem \ref{thm1.2} for global uniqueness in real analytic setting.

\addvspace{10mm}

\section{Symbols of the pseudodifferential operators}\label{s2}

\addvspace{5mm}

Let $(M,g)$ be a compact Riemannian manifold of dimension $n$ with smooth boundary $\partial M$. In the local coordinates $\{x_j\}_{j=1}^n$, we denote by $\bigl\{\frac{\partial}{\partial x_j}\bigr\}_{j=1}^n$ and $\{dx_j\}_{j=1}^n$, respectively, the natural basis for the tangent space $T_x M$ and the cotangent space $T_x^{*} M$ at the point $x\in M$. In what follows, we will use the Einstein summation convention. The Greek indices run from 1 to $n-1$, whereas the Roman indices run from 1 to $n$, unless otherwise specified. Then the Riemannian metric $g$ is given by $g = g_{jk} \,dx_j\otimes dx_k$. 

Let $\nabla_j=\nabla_{\frac{\partial}{\partial x_j}}$ be the covariant derivative with respect to $\frac{\partial}{\partial x_j}$ and $\nabla^j= g^{jk} \nabla_k$. Then for displacement vector field $\bm{u}$, we denote by $\operatorname{div}$ the divergence operator, i.e.,
\begin{align}\label{0.01}
    \operatorname{div}\bm{u}=\nabla_j u^j = \frac{\partial u^j}{\partial x_j} + \Gamma^j_{jk} u^k,\quad \bm{u}=u^j\frac{\partial}{\partial x_j}\in \mathfrak{X}(M).
\end{align}
Here the Christoffel symbols 
\begin{align*}
    \Gamma^{m}_{jk} = \frac{1}{2} g^{ml} \Big(\frac{\partial g_{jl}}{\partial x_k} + \frac{\partial g_{kl}}{\partial x_j} - \frac{\partial g_{jk}}{\partial x_l}\Big),
\end{align*}
and $(g^{jk}) = (g_{jk})^{-1}$. For smooth function $f \in C^{\infty}(M)$, the gradient operator is given by
\begin{align}\label{0.02}
    \operatorname{grad} f = \nabla^j f \frac{\partial}{\partial x_j}
    = g^{jk} \frac{\partial f}{\partial x_j} \frac{\partial}{\partial x_k},\quad f \in C^{\infty}(M).
\end{align}
The Laplace--Beltrami operator is given by
\begin{equation}\label{2.7}
    \Delta_{g} f
    = g^{jk}\Bigl(\frac{\partial^2 f}{\partial x_j \partial x_k} - \Gamma^l_{jk}\frac{\partial f}{\partial x_l}\Bigr), \quad f \in C^{\infty}(M).
\end{equation}
The Lam\'{e} operator \eqref{1.3} with variable coefficients can be rewritten as (see \cite{TanLiu22})
\begin{align}\label{1.6}
    (L_g \bm{u})^j 
    &= \mu \Delta_{g}u^j + (\lambda + \mu)\nabla^j \nabla_k u^k + (\nabla^j \lambda) \nabla_k u^k + (\nabla^k \mu)(\nabla_k u^j + \nabla^j u_k) \notag\\
    &\quad + \mu g^{kl} \Bigl( 2\Gamma^j_{km} \frac{\partial u^m}{\partial x_l} + \frac{\partial \Gamma^j_{kl}}{\partial x_m} u^m \Bigr),\quad j=1,2,\dots,n.
\end{align}

\addvspace{3mm}

Here we briefly introduce the construction of geodesic coordinates with respect to the boundary $\partial M$ (see \cite{LeeUhlm89} or \cite[p.\,532]{Taylor11.2}). For each boundary point $x^{\prime} \in \partial M$, let $\gamma_{x^{\prime}}:[0,\varepsilon)\to \bar{M}$ denote the unit-speed geodesic starting at $x^{\prime}$ and normal to $\partial M$. If $ x^{\prime} := \{x_{1}, \ldots, x_{n-1}\}$ are any local coordinates for $\partial M$ near $x_0 \in \partial M$, we can extend them smoothly to functions on a neighborhood of $x_0$ in $\bar{M}$ by letting them be constant along each normal geodesic $\gamma_{x^{\prime}}$. If we then define $x_n$ to be the parameter along each $\gamma_{x^{\prime}}$, it follows easily that $\{x_{1}, \ldots, x_{n}\}$ form coordinates for $\bar{M}$ in some neighborhood of $x_0$, which we call the boundary normal coordinates determined by $\{x_{1}, \ldots, x_{n-1}\}$. In these coordinates $x_n>0$ in $M$, and $\partial M$ is locally characterized by $x_n=0$. A standard computation shows that the metric has the form $g = g_{\alpha\beta} \,dx_{\alpha} \,dx_{\beta} + dx_{n}^{2}$.

\begin{proposition}\label{prop2.1}
    In the boundary normal coordinates, the thermoelastic Dirichlet-to-Neumann map $\Lambda_g$ can be written as
    \begin{align}\label{2.8}
        \Lambda_g  = A\Bigl(-\frac{\partial }{\partial x_n}\Bigr)-D,
    \end{align}
    where
    \begin{align}
        \label{2.9}A&=
        \begin{bmatrix}
            \mu I_{n-1} &0 &0  \\
            0& \lambda+2\mu &0\\
            0& 0& \alpha
        \end{bmatrix},\\[2mm]
        \label{2.9.1}D&=
        \begin{bmatrix}
            0 & \mu \bigl[g^{\alpha\beta}\frac{\partial }{\partial x_\beta}\bigr] & 0\\[2mm]
            \lambda \bigl[\frac{\partial }{\partial x_\beta} + \Gamma^\alpha_{\alpha\beta}\bigr] & \lambda \Gamma^\alpha_{\alpha n} & -\beta\\[2mm]
            0 & 0 & 0
        \end{bmatrix}.
    \end{align}
\end{proposition}

\addvspace{5mm}

\begin{proof}
    This proof is similar to the proof of \cite[Proposition 2.1]{LiuTan22.2}.
\end{proof}

\addvspace{3mm}

In boundary normal coordinates, we write the Laplace--Beltrami operator as
\begin{align}
    \Delta_g
    & = \frac{\partial^2 }{\partial x_n^2} + \Gamma^\alpha_{\alpha n} \frac{\partial }{\partial x_n} + g^{\alpha\beta} \frac{\partial^2}{\partial x_\alpha\partial x_\beta} +
    \Bigl(
        g^{\alpha\beta} \Gamma^\gamma_{\gamma\alpha} + \frac{\partial g^{\alpha\beta}}{\partial x_\alpha}
    \Bigr)
    \frac{\partial }{\partial x_\beta}.
\end{align}
Combining this and \eqref{1.2}, \eqref{1.3}, \eqref{0.01}--\eqref{1.6}, we deduce that (cf. \cite{TanLiu22,LiuTan22.2})
\begin{align}\label{3.07}
    A^{-1} T_g  = I_{n+1}\frac{\partial^2 }{\partial x_n^2} + B \frac{\partial }{\partial x_n} + C,
\end{align}
where $A$ is given by \eqref{2.9}, $B=B_1+B_0$, $C =C_2+C_1+C_0$, and
\begin{align*}
    B_1&=(\lambda+\mu)
    \begin{bmatrix}
        0 & \frac{1}{\mu}\bigl[g^{\alpha\beta}\frac{\partial}{\partial x_{\beta}}\bigr] & 0\\
        \frac{1}{\lambda+2\mu}\bigl[\frac{\partial}{\partial x_{\beta}}\bigr] & 0 & 0\\
        0 & 0 & 0
    \end{bmatrix},\\[2mm]
    B_0&=
    \begin{bmatrix}
        \Gamma^\alpha_{\alpha n}I_{n-1}+2[\Gamma^\alpha_{n\beta}] & 0 & 0 \\[2mm]
        \frac{\lambda+\mu}{\lambda+2\mu}[\Gamma^\alpha_{\alpha\beta}] & \Gamma^\alpha_{n\alpha} & - \frac{\beta}{\lambda+2\mu}\\[2mm]
        0 & \frac{i\omega\beta\theta_0}{\alpha} & \Gamma^\alpha_{n\alpha}
    \end{bmatrix}
    +\begin{bmatrix}
        \frac{1}{\mu} \frac{\partial \mu}{\partial x_{n}} I_{n-1} & \frac{1}{\mu} [\nabla^\alpha \lambda] & 0 \\[2mm]
        \frac{1}{\lambda+2\mu} \big[\frac{\partial \mu}{\partial x_{\beta}}\big] & \frac{1}{\lambda+2\mu} \frac{\partial (\lambda+2\mu)}{\partial x_{n}} & 0\\[2mm]
        0 & 0 & 0
    \end{bmatrix},\\[2mm]
    C_2&=
    \begin{bmatrix}
        (g^{\alpha\beta}\frac{\partial^2 }{\partial x_\alpha \partial x_\beta})I_{n-1} + \frac{\lambda+\mu}{\mu}\bigl[g^{\alpha\gamma}\frac{\partial^2 }{\partial x_\gamma \partial x_\beta}\bigr] & 0 & 0\\
        0 & \frac{\mu}{\lambda+2\mu}g^{\alpha\beta}\frac{\partial^2 }{\partial x_\alpha \partial x_\beta} & 0\\
        0 & 0 & g^{\alpha\beta}\frac{\partial^2 }{\partial x_\alpha \partial x_\beta} 
    \end{bmatrix},\\[2mm]
    C_1&=
    \begin{bmatrix}
        \bigl((g^{\alpha\beta}\Gamma^\gamma_{\alpha\gamma}+\frac{\partial g^{\alpha\beta}}{\partial x_{\alpha}})\frac{\partial }{\partial x_{\beta}}\bigr) I_{n-1} & 0 & 0 \\
        0 & \frac{\mu}{\lambda+2\mu}\bigl(g^{\alpha\beta}\Gamma^\gamma_{\alpha\gamma}+\frac{\partial g^{\alpha\beta}}{\partial x_{\alpha}}\bigr)\frac{\partial }{\partial x_{\beta}} & 0 \\
        0 & 0 & \bigl(g^{\alpha\beta}\Gamma^\gamma_{\alpha\gamma}+\frac{\partial g^{\alpha\beta}}{\partial x_{\alpha}}\bigr)\frac{\partial }{\partial x_{\beta}}
    \end{bmatrix}\\[2mm]
    &\quad +\frac{\lambda+\mu}{\mu}
    \begin{bmatrix}
        \bigl[g^{\alpha\gamma}\Gamma^\rho_{\rho\beta}\frac{\partial }{\partial x_{\gamma}}\bigr] & \bigl[g^{\alpha\gamma}\Gamma^\rho_{\rho n}\frac{\partial }{\partial x_{\gamma}}\bigr] & 0\\
        0 & 0 & 0\\
        0 & 0 & 0
    \end{bmatrix}\\[2mm]
    &\quad +
    \begin{bmatrix}
        2\bigl[g^{\gamma\rho}\Gamma^\alpha_{\rho\beta}\frac{\partial }{\partial x_{\gamma}}\bigr] & 2\bigl[g^{\gamma\rho}\Gamma^\alpha_{\rho n}\frac{\partial }{\partial x_{\gamma}}\bigr] & -\frac{\beta}{\mu}\bigl[g^{\alpha\beta}\frac{\partial }{\partial x_{\beta}}\bigr] \\[2mm]
        \frac{2\mu}{\lambda+2\mu}\bigl[g^{\gamma\rho}\Gamma^n_{\rho\beta}\frac{\partial }{\partial x_{\gamma}}\bigr] & 0 & 0\\[2mm]
        \frac{i\omega\beta\theta_0}{\alpha}\bigl[\frac{\partial }{\partial x_{\beta}}\bigr] & 0 & 0
    \end{bmatrix}\\[2mm]
    &\quad +
    \begin{bmatrix}
        \frac{1}{\mu}(\nabla^\alpha \mu \frac{\partial }{\partial x_{\alpha}}) I_{n-1} + \frac{1}{\mu} \big[\nabla^\alpha \lambda \frac{\partial }{\partial x_{\beta}} + g^{\alpha\gamma} \frac{\partial \mu}{\partial x_{\beta}} \frac{\partial }{\partial x_{\gamma}}\big] & \frac{1}{\mu} \frac{\partial \mu}{\partial x_{n}}\big[g^{\alpha\beta} \frac{\partial }{\partial x_{\beta}}\big] & 0 \\[2mm]
        \frac{1}{\lambda+2\mu} \frac{\partial \lambda}{\partial x_{n}} \big[\frac{\partial }{\partial x_{\beta}}\big] & \frac{1}{\lambda+2\mu} \nabla^{\alpha}\mu \frac{\partial }{\partial x_{\alpha}}& 0 \\[2mm]
        0& 0 & 0 
    \end{bmatrix},\\[2mm]
    C_0 &=(\lambda+\mu)
    \begin{bmatrix}
        \frac{1}{\mu}\bigl[g^{\alpha\gamma}\frac{\partial \Gamma^\rho_{\rho\beta}}{\partial x_\gamma}\bigr] & \frac{1}{\mu}\bigl[g^{\alpha\gamma}\frac{\partial \Gamma^\rho_{\rho n}}{\partial x_\gamma}\bigr] & 0\\[2mm]
        \frac{1}{\lambda+2\mu}\bigl[\frac{\partial \Gamma^\alpha_{\alpha\beta}}{\partial x_n}\bigr] & \frac{1}{\lambda+2\mu}\frac{\partial \Gamma^\alpha_{\alpha n}}{\partial x_n} & 0\\[2mm]
        0 & 0 & 0
    \end{bmatrix}
    +
    \begin{bmatrix}
        \bigl[g^{ml}\frac{\partial \Gamma^\alpha_{ml}}{\partial x_\beta}\bigr] & \bigl[g^{ml}\frac{\partial \Gamma^\alpha_{ml}}{\partial x_n}\bigr] & 0\\[2mm]
        \frac{\mu}{\lambda+2\mu}\bigl[g^{ml}\frac{\partial \Gamma^n_{ml}}{\partial x_\beta}\bigr] & \frac{\mu}{\lambda+2\mu}g^{ml}\frac{\partial \Gamma^n_{ml}}{\partial x_n} & 0\\[2mm]
        0 & 0 & 0
    \end{bmatrix}\\[2mm]
    &\quad +
    \begin{bmatrix}
        \frac{\rho\omega^2}{\mu}I_{n-1} & 0 & 0\\[2mm]
        0 & \frac{\rho\omega^2}{\lambda+2\mu} & 0\\[2mm]
        \frac{i\omega\beta\theta_0}{\alpha}[\Gamma^\alpha_{\alpha\beta}] & \frac{i\omega\beta\theta_0}{\alpha}\Gamma^\alpha_{\alpha n} & \frac{i\omega\gamma}{\alpha}
    \end{bmatrix}\\[2mm]
    &\quad +
    \begin{bmatrix}
        \frac{1}{\mu} \big[(\nabla^\alpha \lambda) \Gamma^\gamma_{\beta\gamma} - \frac{\partial \mu}{\partial x_{\gamma}} \frac{\partial g^{\alpha\gamma}}{\partial x_{\beta}}\big] & \frac{1}{\mu} \big[(\nabla^\alpha \lambda) \Gamma^\beta_{\beta n} - \frac{\partial \mu}{\partial x_{\beta}} \frac{\partial g^{\alpha\beta}}{\partial x_{n}}\big]  & 0\\[2mm]
        \frac{1}{\lambda+2\mu} \frac{\partial \lambda}{\partial x_{n}} [\Gamma^\alpha_{\alpha\beta}] & \frac{1}{\lambda+2\mu} \frac{\partial \lambda}{\partial x_{n}} \Gamma^\alpha_{\alpha n} & 0\\[2mm]
        0 & 0 & 0
    \end{bmatrix}.
\end{align*}

\addvspace{5mm}

We then derive the microlocal factorization of the thermoelastic operator $T_g$.
\begin{proposition}\label{prop3.1}
    There exists a pseudodifferential operator $Q(x,\partial_{x^\prime})$ of order one in $x^\prime$ depending smoothly on $x_n$ such that
    \begin{align*}
        A^{-1}T_g 
        = \Bigl(I_{n+1}\frac{\partial }{\partial x_n} + B - Q\Bigr)\Bigl(I_{n+1}\frac{\partial }{\partial x_n} + Q\Bigr)
    \end{align*}
    modulo a smoothing operator. Moreover, let $q(x,\xi^{\prime}) \sim \sum_{j\leqslant 1} q_j(x,\xi^{\prime})$ be the full symbol of $Q(x,\partial_{x^\prime})$. Then
    \begin{align}
        q_1&=|\xi^{\prime}|I_{n+1} + \frac{\lambda+\mu}{\lambda+3\mu}F_1, \label{3.9}\\
        q_{-m-1}&=\frac{1}{2|\xi^{\prime}|}E_{-m} - \frac{\lambda+\mu}{4(\lambda+3\mu)|\xi^{\prime}|^2}(F_2E_{-m}+E_{-m}F_1) \notag\\ 
        &\quad - \frac{(\lambda+\mu)^2}{4(\lambda+3\mu)^2|\xi^{\prime}|^3}F_2E_{-m}F_1, \quad m\geqslant -1, \label{3.1.1}
    \end{align}
    where
    \begin{align}
        \label{3.1}F_1&=
        \begin{bmatrix}
            \frac{1}{|\xi^{\prime}|}[\xi^\alpha\xi_\beta] & i[\xi^\alpha] & 0\\[2mm]
            i[\xi_\beta] & -|\xi^{\prime}| & 0\\[2mm]
            0 & 0 & 0
        \end{bmatrix},\\
        \label{3.2}F_2&=
        \begin{bmatrix}
            \frac{1}{|\xi^{\prime}|}[\xi^\alpha\xi_\beta] & -\frac{i(\lambda+2\mu)}{\mu}[\xi^\alpha] & 0 \\[2mm]
            -\frac{i\mu}{\lambda+2\mu}[\xi_\beta] & -|\xi^{\prime}| & 0\\[2mm]
            0 & 0 & 0
        \end{bmatrix},
    \end{align}
    $\xi^\alpha=g^{\alpha\beta}\xi_\beta$, $|\xi^{\prime}|=\sqrt{\xi^\alpha\xi_\alpha}$, $E_1,E_0$ and $E_{-m}\,(m\geqslant 1)$ are given by \eqref{2.11}, \eqref{3.05} and \eqref{4.1}, respectively.
\end{proposition}

\addvspace{5mm}

\begin{proof}
    It follows from \eqref{3.07} that
    \begin{align*}
        I_{n+1}\frac{\partial^2 }{\partial x_n^2} + B \frac{\partial }{\partial x_n} + C 
        = \Bigl(I_{n+1}\frac{\partial }{\partial x_n} + B - Q\Bigr)\Bigl(I_{n+1}\frac{\partial }{\partial x_n} + Q\Bigr).
    \end{align*}
    Equivalently,
    \begin{align}\label{3.7}
        Q^2 - BQ - \Bigl[I_{n+1}\frac{\partial }{\partial x_n},Q\Bigr] + C = 0,
    \end{align}
    where the commutator $\big[I_{n+1}\frac{\partial }{\partial x_n},Q\big]$ is defined by, for any smooth function $f\in C^{\infty}(M)$,
    \begin{align*}
        \Bigl[I_{n+1}\frac{\partial }{\partial x_n},Q\Bigr]f
        &:= I_{n+1}\frac{\partial }{\partial x_n}(Qf) - Q \Bigl(I_{n+1}\frac{\partial }{\partial x_n}\Bigr)f \\
        &= \frac{\partial Q}{\partial x_n}f.
    \end{align*}

    Let $q = q(x,\xi^{\prime})$ be the full symbol of the operator $Q(x,\partial_{x^\prime})$, we write
    \begin{align*}
        q(x,\xi^{\prime}) \sim \sum_{j\leqslant 1} q_j(x,\xi^{\prime})
    \end{align*}
    with $q_j(x,\xi^{\prime})$ homogeneous of degree $j$ in $\xi^{\prime}$. Let
    \begin{align*}
        b(x,\xi^{\prime})=b_1(x,\xi^{\prime}) + b_0(x,\xi^{\prime})
    \end{align*}
    and
    \begin{align*}
        c(x,\xi^{\prime}) = c_2(x,\xi^{\prime}) + c_1(x,\xi^{\prime}) + c_0(x,\xi^{\prime})
    \end{align*}
    be the full symbols of $B$ and $C$, respectively. We denote by $\xi^\alpha=g^{\alpha\beta}\xi_\beta$ and $|\xi^{\prime}|=\sqrt{\xi^\alpha\xi_\alpha}$. Thus, we have $b_0=B_0$, $c_0 =C_0$ and
    \begin{align*}
        &b_1=i(\lambda+\mu)
        \begin{bmatrix}
            0 & \frac{1}{\mu}[\xi^\alpha] & 0 \\
            \frac{1}{\lambda+2\mu}[\xi_\beta] & 0 & 0\\
            0 & 0 & 0
        \end{bmatrix},\\[2mm]
        &c_2=-
        \begin{bmatrix}
            |\xi^{\prime}|^2 I_{n-1} + \frac{\lambda+\mu}{\mu}[\xi^\alpha\xi_\beta] & 0 & 0\\
            0 & \frac{\mu}{\lambda+2\mu}|\xi^{\prime}|^2 & 0\\
            0 & 0 & |\xi^{\prime}|^2
        \end{bmatrix},\\[2mm]
        &c_1=i
        \begin{bmatrix}
            \bigl(\xi^\alpha\Gamma^\beta_{\alpha\beta}+\frac{\partial \xi^\alpha}{\partial x_{\alpha}}\bigr) I_{n-1} & 0 & 0 \\
            0 & \frac{\mu}{\lambda+2\mu}\bigl(\xi^\alpha\Gamma^\beta_{\alpha\beta}+\frac{\partial \xi^\alpha}{\partial x_{\alpha}}\bigr) & 0 \\
            0 & 0 & \xi^\alpha\Gamma^\beta_{\alpha\beta}+\frac{\partial \xi^\alpha}{\partial x_{\alpha}}
        \end{bmatrix}\\[2mm]
        &\qquad +\frac{i(\lambda+\mu)}{\mu}
        \begin{bmatrix}
            [\xi^\alpha\Gamma^\gamma_{\gamma\beta}] & \Gamma^\beta_{\beta n}[\xi^\alpha] & 0\\
            0 & 0 & 0\\
            0 & 0 & 0
        \end{bmatrix}
        +
        \begin{bmatrix}
            2i[\xi^\gamma\Gamma^\alpha_{\gamma\beta}] & 2i[\xi^\gamma\Gamma^\alpha_{\gamma n}] & -\frac{i\beta}{\mu}[\xi^{\alpha}] \\[2mm]
            \frac{2i\mu}{\lambda+2\mu}[\xi^\gamma\Gamma^n_{\gamma\beta}] & 0 & 0\\[2mm]
            -\frac{\omega\beta\theta_0}{\alpha}[\xi_{\beta}] & 0 & 0
        \end{bmatrix}\\[2mm]
        &\qquad + i
        \begin{bmatrix}
            \frac{1}{\mu} (\xi_{\alpha} \nabla^\alpha \mu) I_{n-1} + \frac{1}{\mu} \big[\xi_{\beta}\nabla^\alpha \lambda  + \xi^\alpha \frac{\partial \mu}{\partial x_{\beta}}\big]& \frac{1}{\mu} \frac{\partial \mu}{\partial x_{n}} [\xi^\alpha]  & 0\\[2mm]
            \frac{1}{\lambda+2\mu} \frac{\partial \lambda}{\partial x_{n}} [\xi_{\beta}] & \frac{1}{\lambda+2\mu} \xi_{\alpha} \nabla^{\alpha}\mu & 0\\[2mm]
            0 & 0 & 0
        \end{bmatrix}.
    \end{align*} 
    
    \addvspace{3mm}
    
    Hence, we get the following full symbol equation of \eqref{3.7}
    \begin{equation}\label{3.8}
        \sum_{J} \frac{(-i)^{|J|}}{J !} \partial_{\xi^{\prime}}^{J}q \, \partial_{x^\prime}^{J}q - \sum_{J} \frac{(-i)^{|J|}}{J !} \partial_{\xi^{\prime}}^{J}b \, \partial_{x^\prime}^{J}q - \frac{\partial q}{\partial x_n} + c = 0,
    \end{equation}
    where the sum is over all multi-indices $J$. 

    We shall determine $q_j$ recursively so that \eqref{3.8} holds modulo $S^{-\infty}$. Grouping the homogeneous terms of degree two in \eqref{3.8}, one has
    \begin{align*}
        q_1^2-b_1q_1+c_2=0.
    \end{align*}
    By solving the above matrix equation we get the explicit expression \eqref{3.9} for the principal symbol $q_1$ of $Q$. Here we choose that $q_1$ is positive-definite (cf. \cite{Liu19,LiuTan22.2,TanLiu22}).
    
    Grouping the homogeneous terms of degree one in \eqref{3.8}, we get the following Sylvester equation:
    \begin{align}\label{2.10}
        (q_1-b_1)q_0+q_0q_1=E_1,
    \end{align}
    where
    \begin{align}\label{2.11}
        E_1=i\sum_\alpha\frac{\partial (q_1-b_1)}{\partial \xi_\alpha}\frac{\partial q_1}{\partial x_\alpha}+b_0q_1+\frac{\partial q_1}{\partial x_n} - c_1.
    \end{align}

    Grouping the homogeneous terms of degree zero in \eqref{3.8}, we get
    \begin{align}\label{3.04}
        (q_1-b_1)q_{-1}+q_{-1}q_1=E_0,
    \end{align}
    where
    \begin{align}\label{3.05}
        E_0&=i\sum_\alpha\Bigl(\frac{\partial (q_1-b_1)}{\partial \xi_\alpha}\frac{\partial q_0}{\partial x_\alpha}+\frac{\partial q_0}{\partial \xi_\alpha}\frac{\partial q_1}{\partial x_\alpha}\Bigr)+\frac{1}{2}\sum_{\alpha,\beta}\frac{\partial^2q_1}{\partial \xi_\alpha \partial\xi_\beta}\frac{\partial^2q_1}{\partial x_\alpha \partial x_\beta} \notag\\
        &\quad -q_0^2 +b_0q_0 +\frac{\partial q_0}{\partial x_n} - c_0.
    \end{align}
    
    Proceeding recursively, grouping the homogeneous terms of degree $-m\ (m\geqslant 1)$ in \eqref{3.8}, we get
    \begin{align}\label{4.2}
        (q_1-b_1)q_{-m-1}+q_{-m-1}q_1=E_{-m},
    \end{align}
    where
    \begin{align}\label{4.1}
        E_{-m}= b_0q_{-m}+\frac{\partial q_{-m}}{\partial x_n} - i\sum_\alpha\frac{\partial b_1}{\partial \xi_\alpha}\frac{\partial q_{-m}}{\partial x_\alpha} - \sum_{\substack{-m \leqslant j,k \leqslant 1 \\ |J| = j + k + m}} \frac{(-i)^{|J|}}{J !} \partial_{\xi^{\prime}}^{J} q_j\, \partial_{x^\prime}^{J} q_k
    \end{align}
    for $m \geqslant 1$. Using the methods established in \cite{Liu19,LiuTan22.2,TanLiu22} we solve equations \eqref{2.10}, \eqref{3.04} and \eqref{4.2} to obtain $q_{-m-1}$ for $m \geqslant -1$, see \eqref{3.1.1}.
\end{proof}

\addvspace{5mm}

From the above Proposition \ref{prop3.1} we get the full symbol of the pseudodifferential operator $Q$. This implies that we obtain $Q$ on the boundary modulo a smoothing operator.
\begin{proposition}
    In the boundary normal coordinates, the thermoelastic Dirichlet-to-Neumann map $\Lambda_g$ can be represented as
    \begin{align}\label{3.10}
        \Lambda_g  = AQ-D
    \end{align}
    modulo a smoothing operator, where $A$ and $D$ are given by \eqref{2.9} and \eqref{2.9.1}, respectively.
\end{proposition}

\addvspace{5mm}

\begin{proof}
    This proof is similar to the proof of \cite[Proposition 3.2]{LiuTan22.2}.
\end{proof}

\addvspace{10mm}

\section{Determining coefficients on the boundary}\label{s3}

\addvspace{5mm}

In this section we will prove the uniqueness results for the coefficients $\lambda,\mu,\alpha$ and $\beta$ on the boundary by the full symbol of the thermoelastic Dirichlet-to-Neumann map $\Lambda_g$. We first prove Theorem \ref{thm1.3}.

\addvspace{5mm}

\begin{proof}[Proof of Theorem {\rm \ref{thm1.3}}]
    Let $\sigma(\Lambda_g ) \sim \sum_{j\leqslant 1} p_j(x,\xi^{\prime})$ be the full symbol of the thermoelastic Dirichlet-to-Neumann map $\Lambda_g$. According to \eqref{3.10} and \eqref{2.9.1} we have
\begin{align}
    p_1&=Aq_1-d_1,\label{3.01}\\
    p_0&=Aq_0-d_0,\label{3.02}\\
    p_{-m}&=Aq_{-m},\quad m\geqslant 1,\label{3.03}
\end{align}
where $A$ is given by \eqref{2.9} and
\begin{align}\label{3.08}
    d_1=
    \begin{bmatrix}
        0 & i\mu[\xi^\alpha] & 0\\
        i\lambda [\xi_\beta] & 0 & 0\\
        0 & 0 & 0
    \end{bmatrix},\quad
    d_0=
    \begin{bmatrix}
        0 & 0 & 0\\
        \lambda [\Gamma^\alpha_{\alpha\beta}] & \lambda \Gamma^\alpha_{\alpha n} & -\beta\\
        0 & 0 & 0
    \end{bmatrix}.
\end{align}
Therefore, it is easy to obtain \eqref{18}--\eqref{20}.
\end{proof}

\addvspace{5mm}

We then prove the uniqueness of the coefficients on the boundary.
\begin{proof}[Proof of Theorem {\rm \ref{thm1.1}}]
It follows from \eqref{18}--\eqref{20} that the Lam\'{e} coefficients $\lambda$ and $\mu$ only appear in the $n\times n$ submatrices. In Lam\'{e} system, the uniqueness of $\frac{\partial^{|J|} \lambda}{\partial x^J}$ and $\frac{\partial^{|J|} \mu}{\partial x^J}$ on the boundary for all multi-indices $J$ have been proved in \cite{TanLiu22}. Clearly, this particular result also holds in thermoelastic system and the proof is the same as that of \cite{TanLiu22}. Thus we only need to prove the uniqueness of the coefficients $\alpha$ and $\beta$ on the boundary. 

From \eqref{18} we know that the $(n+1,n+1)$-entry of $p_1$ is
\begin{align*}
    (p_1)^{n+1}_{n+1}=\alpha|\xi^{\prime}|.
\end{align*}
This shows that $p_1$ uniquely determines $\alpha$ on the boundary. Furthermore, the tangential derivatives $\frac{\partial \alpha}{\partial x_\gamma}$ for $1\leqslant \gamma \leqslant n-1$ can also be uniquely determined by $p_1$ on the boundary.

Using the method in \cite{TanLiu22} we solve \eqref{2.10} and obtain
\begin{align*}
    q_0&= \tilde{q_0}+\frac{1}{2|\xi^{\prime}|}E_1^{\prime} - \frac{\lambda+\mu}{4(\lambda+3\mu)|\xi^{\prime}|^2}(F_2E_1^{\prime}+E_1^{\prime}F_1) - \frac{(\lambda+\mu)^2}{4(\lambda+3\mu)^2|\xi^{\prime}|^3}F_2E_1^{\prime}F_1,
\end{align*}
where $\tilde{q_0}$ is the solution of the corresponding equation with constant coefficients (see \cite[p.\,13]{LiuTan22.2}), $F_1$ and $F_2$ are given by \eqref{3.1} and \eqref{3.2}, respectively.
\begin{align*}
    E_1^{\prime}=b_0^{\prime}q_1-c_1^{\prime}.
\end{align*}
Here
\begin{align*}
    b_0^{\prime}=\begin{bmatrix}
        \frac{1}{\mu} \frac{\partial \mu}{\partial x_{n}} I_{n-1} & \frac{1}{\mu} [\nabla^\alpha \lambda] & 0 \\[2mm]
        \frac{1}{\lambda+2\mu} \big[\frac{\partial \mu}{\partial x_{\beta}}\big] & \frac{1}{\lambda+2\mu} \frac{\partial (\lambda+2\mu)}{\partial x_{n}} & 0\\[2mm]
        0 & 0 & 0
    \end{bmatrix}
\end{align*}
and
\begin{align*}
    c_1^{\prime}=i
    \begin{bmatrix}
        \frac{1}{\mu} (\xi_{\alpha} \nabla^\alpha \mu) I_{n-1} + \frac{1}{\mu} \big[\xi_{\beta}\nabla^\alpha \lambda  + \xi^\alpha \frac{\partial \mu}{\partial x_{\beta}}\big]& \frac{1}{\mu} \frac{\partial \mu}{\partial x_{n}} [\xi^\alpha]  & 0\\[2mm]
        \frac{1}{\lambda+2\mu} \frac{\partial \lambda}{\partial x_{n}} [\xi_{\beta}] & \frac{1}{\lambda+2\mu} \xi_{\alpha} \nabla^{\alpha}\mu & 0\\[2mm]
        0 & 0 & 0
    \end{bmatrix}.
\end{align*}
Hence, we see that $q_0$ has the form (see \cite[p.\,13]{LiuTan22.2})
\begin{align}\label{3.09}
    q_0=
    \begin{bmatrix}
        * & * & \frac{i\beta}{(\lambda+3\mu)|\xi^{\prime}|}[\xi_\alpha] \\[2mm]
        * & * & -\frac{\beta}{\lambda+3\mu} \\[2mm]
        \frac{\mu\omega\beta\theta_0}{\alpha(\lambda+3\mu)|\xi^{\prime}|}[\xi_\beta] & \frac{i\mu\omega\beta\theta_0}{\alpha(\lambda+3\mu)} & *
    \end{bmatrix},
\end{align}
where $*$ denotes the terms which we do not care (of course, they can be computed explicitly).

Therefore, combining \eqref{3.09}, \eqref{3.02} and \eqref{3.08} we get the $(n,n+1)$-entry $(p_0)^n_{n+1}$, that is,
\begin{align*}
    (p_0)^n_{n+1}=\beta-\frac{\beta(\lambda+2\mu)}{\lambda+3\mu}=\frac{\beta\mu}{\lambda+3\mu}.
\end{align*}
This implies that $p_0$ uniquely determines $\beta$ on the boundary and the tangential derivatives $\frac{\partial \beta}{\partial x_\gamma}$ on the boundary for $1\leqslant \gamma \leqslant n-1$ since $\lambda$ and $\mu$ have been determined on the boundary by the previous arguments.

According to the above discussion, we see from \eqref{3.02} that $q_0$ is uniquely determined by $p_0$ since the boundary values of $\lambda,\mu,\alpha$ and $\beta$ have been uniquely determined. By \eqref{2.10} we can determine $E_1$ from the knowledge of $q_0$. For $k\geqslant 0$, we denote by $\mathcal{T}_{-k}=\mathcal{T}_{-k}(\lambda,\mu,\alpha,\beta)$ the terms which involve only the boundary values of $\lambda,\mu,\alpha,\beta$ and their normal derivatives of order ar most $k$ (which have been uniquely determined). Note that $\mathcal{T}_{-k}$ may be different in different expressions. 

From \eqref{2.11}, we have
\begin{align}\label{3.06}
    E_1=b_0q_1+\frac{\partial q_1}{\partial x_n}-c_1+\mathcal{T}_0.
\end{align}
By \eqref{3.03} and \eqref{3.04} we know that $q_{-1}$ is uniquely determined by $p_{-1}$, and $E_0$ can be determined from the knowledge of $q_{-1}$. From \eqref{3.05} we see that
\begin{align}\label{3.11}
    E_0=\frac{\partial q_0}{\partial x_n}+\mathcal{T}_{-1}.
\end{align}
From \eqref{3.09} we find that the $(n,n+1)$-entry $(\frac{\partial q_0}{\partial x_n})^n_{n+1}$ and the $(n+1,n)$-entry $(\frac{\partial q_0}{\partial x_n})^{n+1}_n$ of $\frac{\partial q_0}{\partial x_n}$ are, respectively,
\begin{align}
    \label{3.12} \Big(\frac{\partial q_0}{\partial x_n}\Big)^n_{n+1} &= -\frac{\frac{\partial \beta}{\partial x_n}(\lambda+3\mu)-\beta(\frac{\partial \lambda}{\partial x_n}+3\frac{\partial \mu}{\partial x_n})}{(\lambda+3\mu)^2} \notag\\
    &=-\frac{1}{\lambda+3\mu}\frac{\partial \beta}{\partial x_n}+\mathcal{T}_{-1},\\
    \label{3.13} \Big(\frac{\partial q_0}{\partial x_n}\Big)^{n+1}_n &= \frac{-\beta\mu(\lambda+3\mu)\frac{\partial \alpha}{\partial x_n}+\alpha\mu(\lambda+3\mu)\frac{\partial \beta}{\partial x_n}+\alpha\beta(\lambda\frac{\partial \mu}{\partial x_n}-\mu\frac{\partial \lambda}{\partial x_n})}{\alpha^2(\lambda+3\mu)^2} \notag\\
    &=-\frac{\beta\mu}{\alpha^2(\lambda+3\mu)}\frac{\partial \alpha}{\partial x_n} + \frac{\mu}{\alpha(\lambda+3\mu)}\frac{\partial \beta}{\partial x_n}+\mathcal{T}_{-1}.
\end{align}
Since $\alpha,\beta,\lambda,\mu,\frac{\partial \lambda}{\partial x_n}$ and $\frac{\partial \mu}{\partial x_n}$ have been determined on the boundary, then $\frac{\partial \beta}{\partial x_n}$ can be determined by $(\frac{\partial q_0}{\partial x_n})^n_{n+1}$ on the boundary, and $\frac{\partial \alpha}{\partial x_n}$ can be determined by $(\frac{\partial q_0}{\partial x_n})^{n+1}_n$ on the boundary. This implies that $p_{-1}$ uniquely determines $\frac{\partial \alpha}{\partial x_n}$ and $\frac{\partial \beta}{\partial x_n}$ on the boundary.

By \eqref{2.10} we have
\begin{align*}
    (q_1-b_1)\frac{\partial q_0}{\partial x_n}+\frac{\partial q_0}{\partial x_n}q_1=\frac{\partial E_1}{\partial x_n}+\mathcal{T}_{-1}.
\end{align*}
This implies that $\frac{\partial E_1}{\partial x_n}$ can be determined from the knowledge of $\frac{\partial q_0}{\partial x_n}$. By \eqref{3.03} and \eqref{4.2} we know that $q_{-2}$ is uniquely determined by $p_{-2}$, and $E_{-1}$ can be determined from the knowledge of $q_{-2}$. From \eqref{4.1} we see that
\begin{align*}
    E_{-1}=\frac{\partial q_{-1}}{\partial x_n}+\mathcal{T}_{-2}.
\end{align*}
By \eqref{4.2} we have
\begin{align*}
    (q_1-b_1)\frac{\partial q_{-1}}{\partial x_n}+\frac{\partial q_{-1}}{\partial x_n}q_1=\frac{\partial E_0}{\partial x_n}+\mathcal{T}_{-2}.
\end{align*}
This implies that $\frac{\partial E_0}{\partial x_n}$ can be determined from the knowledge of $\frac{\partial q_{-1}}{\partial x_n}$. From \eqref{3.11} we have
\begin{align*}
    \frac{\partial E_0}{\partial x_n}=\frac{\partial^2 q_0}{\partial x_n^2} + \mathcal{T}_{-2}.
\end{align*}
Thus, it follows from \eqref{3.12} and \eqref{3.13} that
\begin{align*}
    \Big(\frac{\partial^2 q_0}{\partial x_n^2}\Big)^n_{n+1} 
    &=-\frac{1}{\lambda+3\mu}\frac{\partial^2 \beta}{\partial x_n^2}+\mathcal{T}_{-2},\\
    \Big(\frac{\partial^2 q_0}{\partial x_n^2}\Big)^{n+1}_n 
    &=-\frac{\beta\mu}{\alpha^2(\lambda+3\mu)}\frac{\partial^2 \alpha}{\partial x_n^2} + \frac{\mu}{\alpha(\lambda+3\mu)}\frac{\partial^2 \beta}{\partial x_n^2}+\mathcal{T}_{-2}.
\end{align*}
Since $\lambda,\mu,\alpha,\beta,\frac{\partial \lambda}{\partial x_n},\frac{\partial \mu}{\partial x_n},\frac{\partial^2 \lambda}{\partial x_n^2},\frac{\partial^2 \mu}{\partial x_n^2},\frac{\partial \alpha}{\partial x_n}$ and $\frac{\partial \beta}{\partial x_n}$ have been determined on the boundary, then $\frac{\partial^2 \beta}{\partial x_n^2}$ can be determined by $(\frac{\partial^2 q_0}{\partial x_n^2})^n_{n+1}$ on the boundary, and $\frac{\partial^2 \alpha}{\partial x_n^2}$ can be determined by $(\frac{\partial^2 q_0}{\partial x_n^2})^{n+1}_n$ on the boundary. This implies that $p_{-2}$ uniquely determines $\frac{\partial \alpha^2}{\partial x_n^2}$ and $\frac{\partial^2 \beta}{\partial x_n^2}$ on the boundary.

Finally, we consider $p_{-m-1}$ for $m\geqslant 1$. By \eqref{3.03} and \eqref{4.2} we have $p_{-m-1}$ uniquely determines $q_{-m-1}$, and $E_{-m}$ can be determined from the knowledge of $q_{-m-1}$. From \eqref{4.1} we obtain
\begin{align*}
    E_{-m}=\frac{\partial q_{-m}}{\partial x_n}+\mathcal{T}_{-m-1}.
\end{align*}
We see from \eqref{4.2} that
\begin{align*}
    (q_1-b_1)\frac{\partial q_{-m}}{\partial x_n}+\frac{\partial q_{-m}}{\partial x_n}q_1=\frac{\partial E_{-m+1}}{\partial x_n}+\mathcal{T}_{-m-1}.
\end{align*}
This implies that $\frac{\partial E_{-m+1}}{\partial x_n}$ can be determined from the knowledge of $\frac{\partial q_{-m}}{\partial x_n}$. 

We end this proof by induction. Suppose we have shown that, by iteration, $E_{-m}$ uniquely determines
\begin{align}\label{3.14}
    \frac{\partial^m E_0}{\partial x_n^m}=\frac{\partial^{m+1} q_0}{\partial x_n^{m+1}}+\mathcal{T}_{-m-1},
\end{align}
which further determines $\frac{\partial^{m+1} \alpha}{\partial x_n^{m+1}}$ and $\frac{\partial^{m+1} \beta}{\partial x_n^{m+1}}$ on the boundary since we have
\begin{align*}
    \Big(\frac{\partial^{m+1} q_0}{\partial x_n^{m+1}}\Big)^n_{n+1} 
    &=-\frac{1}{\lambda+3\mu}\frac{\partial^{m+1} \beta}{\partial x_n^{m+1}}+\mathcal{T}_{-m-1},\\
    \Big(\frac{\partial^{m+1} q_0}{\partial x_n^{m+1}}\Big)^{n+1}_n 
    &=-\frac{\beta\mu}{\alpha^2(\lambda+3\mu)}\frac{\partial^{m+1} \alpha}{\partial x_n^{m+1}} + \frac{\mu}{\alpha(\lambda+3\mu)}\frac{\partial^{m+1} \beta}{\partial x_n^{m+1}}+\mathcal{T}_{-m-1}.
\end{align*}

By \eqref{3.03} and \eqref{4.2} we know that $q_{-m-2}$ is uniquely determined by $p_{-m-2}$, and $E_{-m-1}$ can be determined from the knowledge of $q_{-m-2}$. Hence, $E_{-m-1}$ uniquely determines $\frac{\partial^{m+2} q_0}{\partial x_n^{m+2}}$ by iteration. It follows that
\begin{align*}
    \Big(\frac{\partial^{m+2} q_0}{\partial x_n^{m+2}}\Big)^n_{n+1} 
    &=-\frac{1}{\lambda+3\mu}\frac{\partial^{m+2} \beta}{\partial x_n^{m+2}}+\mathcal{T}_{-m-2},\\
    \Big(\frac{\partial^{m+2} q_0}{\partial x_n^{m+2}}\Big)^{n+1}_n 
    &=-\frac{\beta\mu}{\alpha^2(\lambda+3\mu)}\frac{\partial^{m+2} \alpha}{\partial x_n^{m+2}} + \frac{\mu}{\alpha(\lambda+3\mu)}\frac{\partial^{m+2} \beta}{\partial x_n^{m+2}}+\mathcal{T}_{-m-2}.
\end{align*}
This implies that $p_{-m-2}$ uniquely determines $\frac{\partial^{m+2} \alpha}{\partial x_n^{m+2}}$ and $\frac{\partial^{m+2} \beta}{\partial x_n^{m+2}}$ on the boundary.

Therefore, by combining the uniqueness result of $\frac{\partial^{|J|} \lambda}{\partial x^J}$, $\frac{\partial^{|J|} \mu}{\partial x^J}$ (see \cite{TanLiu22}) and the above arguments, we conclude that the thermoelastic Dirichlet-to-Neumann map $\Lambda_g$ uniquely determines $\frac{\partial^{|J|} \lambda}{\partial x^J}$, $\frac{\partial^{|J|} \mu}{\partial x^J}$, $\frac{\partial^{|J|} \alpha}{\partial x^J}$ and $\frac{\partial^{|J|} \beta}{\partial x^J}$ on the boundary for all multi-indices $J$.
\end{proof}

\addvspace{10mm}

\section{Global uniqueness of real analytic coefficients}\label{s4}

\addvspace{5mm}

This section is devoted to proving the global uniqueness of real analytic coefficients $\lambda,\mu,\alpha$ and $\beta$ on a real analytic manifold. More precisely, we prove that the thermoelastic Dirichlet-to-Neumann map $\Lambda_g$ uniquely determines the real analytic coefficients on the whole manifold $\bar{M}$. 

We recall that the definitions of real analytic functions and real analytic hypersurfaces of a Riemannian manifold. Let $f(x)$ be a real-valued function defined on an open set $\Omega\subset\mathbb{R}^n$. For $y\in\Omega$ we call $f(x)$ real analytic at $y$ if there exist $a_J \in\mathbb{R}$ and a neighborhood $N_y$ of $y$ such that
\begin{align*}
    f(x)=\sum_{J} a_J(x-y)^J
\end{align*}
for all $x\in N_y$ and $J\in\mathbb{N}^n$. We say $f(x)$ is real analytic on an open set $\Omega$ if $f(x)$ is real analytic at each $y\in\Omega$.

Let $(M,g)$ be a Riemannian manifold. A subset $U$ of $M$ is said to be an $(n-1)$-dimensional real analytic hypersurface if $U$ is nonempty and if for every point $x\in U$, there is a real analytic diffeomorphism of an unit open ball $B(0,1)\subset\mathbb{R}^n$ onto an open neighborhood $N_x$ of $x$ such that $B(0,1)\cap\{x\in\mathbb{R}^n|x_n = 0\}$ maps onto $N_x\cap U$.

\addvspace{3mm}

In order to prove Theorem \ref{thm1.2}, we need the following lemma (see \cite[p.\,65]{John82}).
\begin{lemma}\label{lem3.1}
    $($Unique continuation of real analytic functions$)$ Let $ M\subset\mathbb{R}^n$ be a connected open set and $f(x)$ be a real analytic function defined on $ M$. Let $y\in M$. Then $f(x)$ is uniquely determined in $ M$ if we know $\frac{\partial^{|J|}f(y)}{\partial x^J}$ for all $J\in\mathbb{N}^n$. In particular, $f(x)$ is uniquely determined in $ M$ by its values in any nonempty open subset of $ M$.
\end{lemma}

Note that Lemma \ref{lem3.1} still holds for real analytic functions defined on real analytic manifolds. Finally, we prove Theorem \ref{thm1.2}.

\addvspace{5mm}

\begin{proof}[Proof of Theorem {\rm \ref{thm1.2}}]
    According to Theorem \ref{thm1.1}, it has been proved that the thermoelastic Dirichlet-to-Neumann map $\Lambda_g$ uniquely determines $\frac{\partial^{|J|} \lambda}{\partial x^J}$, $\frac{\partial^{|J|} \mu}{\partial x^J}$, $\frac{\partial^{|J|} \alpha}{\partial x^J}$ and $\frac{\partial^{|J|} \beta}{\partial x^J}$ on the boundary for all multi-indices $J$. Hence, for any point $x_0\in\Gamma$, the coefficients can be uniquely determined in some neighborhood of $x_0$ by the analyticity of the coefficients on $M\cup\Gamma$. Furthermore, it follows from Lemma \ref{lem3.1} that the coefficients can be uniquely determined in $M$. Therefore, by combining Theorem \ref{thm1.1} we conclude that the coefficients $\lambda,\mu,\alpha$ and $\beta$ can be uniquely determined on $\bar{M}$ by the thermoelastic Dirichlet-to-Neumann map $\Lambda_g$.
\end{proof}

\addvspace{5mm}

\begin{remark}
    By applying the method of Kohn and Vogelius {\rm \cite{KohnVoge85}}, we can also prove that the thermoelastic Dirichlet-to-Neumann map $\Lambda_g$ uniquely determines the coefficients $\lambda,\mu,\alpha$ and $\beta$ on $\bar{M}$ provided the manifold and the coefficients are piecewise analytic.
\end{remark}

\addvspace{10mm}

\section*{Acknowledgements}

\addvspace{5mm}

This research was supported by National Natural Science Foundation of China (No. 12271031) and National Key Research and Development Program of China (No. \\2022YFC3310300).

\addvspace{10mm}

\addvspace{5mm}

\end{document}